\documentclass[12pt,english]{amsart}
\usepackage[cp1251]{inputenc}
\usepackage[english]{babel}
\usepackage{amsmath,amsthm,amssymb}
\newtheorem{theorem}{Theorem}[section]
\newtheorem{lemma}[theorem]{Lemma}
\newtheorem{corollary}[theorem]{Corollary}
\newtheorem{remark}[theorem]{Remark}

\theoremstyle{definition}

\begin{document}

\title[]{Regularity of linear and polynomial images of Skorohod differentiable measures}

\author{Egor D. Kosov}

\begin{abstract}
In this paper we study the regularity properties of
linear and polynomial images of Skorohod differentiable measures.
Firstly, we obtain estimates for the Skorohod derivative norm
of a projection of a product of Scorohod differentiable measures.
In the second part of the paper we prove Nikolskii--Besov regularity
of a polynomial image of a Skorohod differentiable measure on $\mathbb{R}^n$.
\end{abstract}

\maketitle



\section{Introduction}

Let $\mu$ be a Skorohod differentiable probability measure on $\mathbb{R}^n$.
In this paper we study the regularity properties of induced measures
$\mu\circ f^{-1}$, where $f$ is a linear or a polynomial function.

The first part of this paper deals with measures $\mu$ of a special type:
$\mu=\mu_1\otimes\ldots\otimes\mu_n$,
where each $\mu_j$ is a probability measure on $\mathbb{R}$.
Such measures are
distributions of random vectors $\xi:=(\xi_1,\ldots,\xi_n)$
with independent components $\xi_1,\ldots,\xi_n$.
It was  proved in \cite{RV} (see also \cite{BCh})
 that there is an absolute constant $C$ such that, whenever
 $P$ is the orthogonal projection  onto some $k$-dimensional subspace of $\mathbb{R}^n$,
the density of the random vector
$P\xi$ is bounded by~$C^k$, provided that the density of each $\xi_j$ is not greater than $1$.
In particular, in the case $k=1$ one has $C=\sqrt2$ as a corollary of two known results
of Rogozin \cite{Rog} and Ball \cite{B-Cube1}.
In \cite{LPP}, Ball's arguments from \cite{B-Cube1} and \cite{B-Cube2}
were adapted to provide some sharp bounds in the
Rudelson--Vershynin theorem from \cite{RV}.
In this paper we consider a similar question concerning
bounds on the Skorohod derivative norm
of the distribution of the random vector $P\xi$,
provided that the distribution of each $\xi_j$ is
Skorohod differentiable. In particular, we provide an analog
of the Rudelson--Vershynin result for the norm of the
Skorohod derivative (all necessary definitions will be recalled in Section~$2$):

{\it
For each $k\in \mathbb{N}$ there is $C(k)$ such that
for any probability measures $\mu_1,\ldots,\mu_n$ on $\mathbb{R}$
with $\max\limits_{1\le j\le n}\|\mu_j'\|_{\rm TV}\le 1$ and for any
orthogonal projector $P\colon\mathbb{R}^n\to\mathbb{R}^k$
one has
$$
\sup\limits_{|e|=1}\|D_e[\mu\circ P^{-1}]\|_{\rm TV}\le C(k),
$$
where $\mu=\mu_1\otimes\ldots\otimes\mu_n$. Moreover, $C(1)=\sqrt2$.
}

The proof is based on the recent result of Bobkov, Chistyakov and G\"otze
\cite[Lemma 4.3]{BChG}
about representation of any Skorohod differentiable probability measure $\nu$
on $\mathbb{R}$ as a convex mixture of uniform distributions
$$
\nu = \int \nu_{[a,b]}\, \pi(dadb)
$$
such that
$$
\|\nu'\|_{\rm TV}=\int \|\nu_{[a,b]}'\|_{\rm TV}\, \pi(dadb).
$$
This representation enables one to reduce the proof to the case
where $\xi=(\xi_1,\ldots,\xi_n)$ is uniformly distributed
in the cube $Q_n=[-1/2, 1/2]^n$ (i.e., $\mu(dx)=I_{Q_n}(x)dx$)
and than apply some known results from the theory of logarithmically
concave measures, since any linear image of the uniform distribution
on a convex set is a logarithmically concave measure.

The second main result of this paper concerns polynomial images of
general Skorohod differentiable probability measures $\mu$ and asserts that,
for any polynomial $f$ of degree $d$ on $\mathbb{R}^n$,
the distribution density of $\mu\circ f^{-1}$ belongs to the Nikolskii--Besov class
$B^{1/d}_{1,\infty}(\mathbb{R})$.
The Nikolskii--Besov fractional regularity of
polynomial images of measures
was studied in \cite{Kos} and \cite{BKZ}
in the case of Gaussian and general logarithmically concave measures.
In particular, it was proved in~\cite{Kos} that for any log-concave measure and any polynomial $f$
of degree $d$ on $\mathbb{R}^n$ one has
$$
[\mathbb{D}f]^{1/d}
\int_\mathbb{R}|\rho_f(t+h) - \rho_f(t)|\, dt\le C(d)|h|^{1/d} \quad \forall h\in \mathbb{R},
$$
where $\rho_f$ is the density of the measure
$\mu\circ f^{-1}$ and  $\mathbb{D}f$ is the variance of $f$.
In this paper we obtain a similar bound in the case where $\mu$
is a general Skorohod differentiable probability measure:
$$
\|B_d\|^{1/d}\int_\mathbb{R}|\rho_f(t+h) - \rho_f(t)|\, dt
\le 24\pi\sup_{|\theta|=1}\|D_\theta\mu\|_{\rm TV}|h|^{1/d},
$$
where
$f(x):= \sum\limits_{j=0}^d B_j(x,\ldots,x)$,
each $B_j\colon (\mathbb{R}^n)^j\to \mathbb{R}$ is a symmetric $j$-linear
function on $(\mathbb{R}^n)^j$ and
$$
\|B_d\|:=\sup_{|x_1|=1,\ldots, |x_d|=1}|B_d(x_1, \ldots, x_d)|.
$$
The proof is based on some ideas of the papers \cite{Kos} and \cite{BKZ}
and relies on several lemmas from these papers.

\section{Preliminaries}

We first recall some definitions and notation which will be used throughout this paper.

Let $\langle\cdot,\cdot\rangle$ denote the standard inner product on $\mathbb{R}^n$
and let $|\cdot|$ be the usual norm generated by this inner product.

The total variation norm of a (signed) measure $\sigma$ on $\mathbb{R}^n$
is defined by the equality
$$
\|\sigma\|_{\rm TV} :=
\sup\biggl\{\int \varphi\, d\sigma, \ \varphi\in C_0^\infty(\mathbb{R}^n),\ \|\varphi\|_\infty\le1\biggr\},
$$
where $\|\varphi\|_\infty:=\sup|\varphi(x)|$
and  $C_0^\infty(\mathbb{R}^n)$ is the space
of all infinitely differentiable functions with compact support.

Let $\mu$ be a Borel probability  measure on $\mathbb{R}^n$.
The measure $\mu$ is called Skorohod differentiable
along $h\in\mathbb{R}^n$ if
there is a bounded signed measure $D_h\mu$
such that
$$
\int_{\mathbb{R}^n} \partial_h \varphi\, d\mu=
-\int_{\mathbb{R}^n} \varphi\, d(D_h\mu)
$$
for every $\varphi\in C_0^\infty(\mathbb{R}^n)$ (see \cite{DiffMeas}).
A measure $\mu$ has the Skorohod derivatives along all vectors~$h$ precisely when
it possesses a density $\rho_\mu$ of class $BV$ (the class of functions of bounded variation).
The latter means that every
generalized derivative $\partial_h\rho_\mu$ is a bounded measure.
In the one-dimensional case in place of $D_1\mu$ we write $\mu'$.
For a $\mu$-measurable mapping $f\colon\mathbb{R}^n\to\mathbb{R}^k$,
the image measure $\mu\circ f^{-1}$
is defined by the equality
$$
\mu\circ f^{-1}(A):=\mu(f\in A)$$
for all Borel sets $A\subset\mathbb{R}^k$.

A Borel probability  measure $\mu$ on $\mathbb{R}^n$ is
logarithmically concave (log-concave or convex) if
$$
    \mu(t A + (1-t)B) \ge \mu(A)^{t}\mu(B)^{1-t} \quad \forall\, t\in [0,1]
$$
for every pair of Borel sets $A,B$ (see \cite{Bor}).
Equivalently, $\mu$
has a density of the form $e^{-V}$ with respect to Lebesgue measure
on some affine subspace~$L$,
where $V\colon\, L\to (-\infty, +\infty]$ is a convex function.
Two main examples of log-concave measures are
uniform distributions on convex sets and Gaussian measures.
A log-concave measure $\mu$ on $\mathbb{R}^n$ is called isotropic
if it is absolutely continuous with respect to Lebesgue measure and
$$
\int_{\mathbb{R}^n}\langle x,\theta\rangle\, \mu(dx)=0,\quad
\int_{\mathbb{R}^n}\langle x,\theta\rangle^2\, \mu(dx)=
|\theta|^2\quad \forall\ \theta\in\mathbb{R}^n.
$$
It is known (see \cite{Krug} and also \cite[Section~4.3]{DiffMeas})
that
for any log-concave measure $\mu$ with a density $\rho$
and any unit vector $h$ one has
$$
    \|D_h\mu\|_{\rm TV}=2\int_{\langle h\rangle^\bot} \max\limits_t\rho(x+th)\, dx,
$$
where $\langle h\rangle^\bot$ is the orthogonal complement of $h$.

Recall that a function $f$ on $\mathbb{R}^n$ is a polynomial of degree $d$ if it is of the form
$$
f(x):= \sum_{j=0}^d B_j(x,\ldots,x),
$$
where each $B_j\colon (\mathbb{R}^n)^j\to \mathbb{R}$ is a symmetric $j$-linear
function on $(\mathbb{R}^n)^j$ and $B_d$ is not identically zero.
For a multilinear function $B\colon (\mathbb{R}^n)^d\to\mathbb{R}$ set
$$
\|B\|:=\sup_{|x_1|=1,\ldots, |x_d|=1}|B(x_1, \ldots, x_d)|,
$$
where $x_1, \ldots, x_d\in\mathbb{R}^n$.

For a function $g$ on $\mathbb{R}$ we set
$$
\|g\|_p := \Bigl(\int_{\mathbb{R}} |g(x)|^p\, dx\Bigr)^{1/p},\quad \|g\|_\infty:= \sup_{x\in\mathbb{R}}|g(x)|.
$$
Recall (see \cite{BIN}, \cite{Nikol77}, \cite{Trieb}, \cite{Stein}) that the Nikolskii--Besov class
$B^\alpha_{1,\infty}(\mathbb{R})$ with $\alpha\in(0,1)$ consists of all functions $g\in L^1(\mathbb{R})$
such that
$$
\int_\mathbb{R}|g(x+h)-g(x)|\, dx\le C|h|^\alpha
$$
for some constant $C>0$.

\vskip .1in

We now formulate several known results that will be used in the proofs.

\begin{theorem}[see \cite{Klartag}, \cite{Ball}]\label{t1.1}
For every $k\in \mathbb{N}$, there is $C_k$ depending only on  $k$ such that
for every isotropic log-concave measure $\mu$ on $\mathbb{R}^k$
with density $\rho$
one has
$$
(\max\rho)^{1/k}\le C_k.
$$
\end{theorem}

There is an open conjecture
(the so-called hyperplane conjecture) that the constant above
can be chosen independent of dimension $k$. However,
the best known constant so far is  $C_k\sim k^{1/4}$, which  is due to Klartag \cite{Klartag}.

The following result is Corollary 2.4 in \cite{Klartag07}.

\begin{theorem}[see \cite{Klartag07}]\label{t1.2}
For every $k\in\mathbb{N}$,
there are universal constants $C, c>0$
such that, for every
isotropic log-concave measure $\mu$ on $\mathbb{R}^k$
with density $\rho$, one has
$$
\rho(x)\le \rho(0) e^{Ck - c|x|} \quad \forall x\in\mathbb{R}^k.
$$
\end{theorem}

We say that a probability measure $\nu$ on $\mathbb{R}$
is represented as a convex mixture of uniform distributions
with a mixing measure $\pi$ if $\pi$ is a probability measure
on the half plane $\{a<b\}$ such that
$$
\int_\mathbb{R}\varphi\, d\nu = \int_{\{a<b\}} \tfrac{1}{b-a}\int_a^b\varphi(x)\, dx\, \pi(dadb)
\quad \forall \varphi\in C_0^\infty(\mathbb{R}).
$$
For simplicity we denote the uniform distribution in the interval $[a,b]$
by $\nu_{[a,b]}$.

The following theorem is Lemma 4.3 in \cite{BChG}.
This theorem plays a crucial role in the proof of the first main result
of this paper.

\begin{theorem}[see \cite{BChG}]\label{t1.3}
Any Skorohod differentiable probability measure $\nu$ on
$\mathbb{R}$ can be represented as a convex
mixture
$$
\nu=\int \nu_{[a,b]}\, \pi(dadb)
$$
of uniform distributions with a mixing measure $\pi$
such that
$$
\|\nu'\|_{\rm TV}=\int \|\nu_{[a,b]}'\|_{\rm TV}\, \pi(dadb).
$$
\end{theorem}

We recall that $\|\nu_{[a,b]}'\|_{\rm TV} = \tfrac{2}{b-a}$.

We now present two key lemmas from \cite{Kos} and \cite{BKZ}
that will be used in the second part of the paper.
The first one provides a sufficient condition
for a measure on $\mathbb{R}$ to possess a density from a
certain Nikolskii--Besov class
(see also \cite{BKP} and \cite{KosBes} for further development of this approach).

\begin{lemma}[see \cite{Kos}, \cite{BKZ}]\label{lem1.1}
Let $\nu$ be a Borel probability  measure on the real line.
Assume that for every function
$\varphi\in C_0^\infty(\mathbb{R})$ with $\|\varphi\|_\infty\le1$
one has
$$
\int\varphi'd\nu\le C\|\varphi'\|_\infty^{1-\alpha}.
$$
Then
$$
\|\nu_h-\nu\|_{\rm TV}\le 2^{1-\alpha}C|h|^\alpha\quad \forall\, h\in\mathbb{R},
$$
where $\nu_h(A):=\nu (A-h)$ for all Borel sets $A\subset \mathbb{R}$.
In particular, the density of the measure $\nu$ belongs to the Nikolskii--Besov class
$B^\alpha_{1,\infty}(\mathbb{R})$.
\end{lemma}

We also need the following result.

\begin{lemma}[see \cite{Kos}]\label{lem1.2}
Let $\nu$ be a Borel probability measure on the real line.
Assume that for every function $\varphi\in C_0^\infty(\mathbb{R})$ with $\|\varphi\|_\infty\le1$
one has
$$
\int\varphi'd\nu\le C\|\varphi'\|_\infty^{1-\alpha}.
$$
Then, for every Borel set $A$, one has
$$
\nu(A)\le C\lambda(A)^\alpha,
$$
where $\lambda$ is the standard Lebesgue measure on the real line.
Moreover, if $\rho_\nu$ is the density of the measure $\nu$,
then $\rho_\nu\in L^p(\lambda)$ whenever $1<p<\frac{1}{1-\alpha}$ and one has
$$
\|\rho_\nu\|_{L^p(\lambda)}\le \biggl(p + p\Bigl(\frac{1}{1-\alpha}-p\Bigr)^{-1}\biggr)^{1/p}C^{\frac{1}{\alpha}(1-1/p)}.
$$
\end{lemma}

The second part of this lemma is the usual embedding theorem
for Nikolskii--Besov spaces.

\section{Linear images of products of Skorohod differentiable measures}

We start with the one-dimensional case.

\begin{theorem}
Let $\mu_1,\ldots, \mu_n$ be Skorohod differentiable
probability measures on $\mathbb{R}$. Then,
for any linear functional $f(x)=\langle a, x\rangle$, where $a=(a_1,\ldots, a_n)\in \mathbb{R}^n$
with $|a|=1$,
one has
$$
\|(\mu\circ f^{-1})'\|_{\rm TV}\le \sqrt2 \max\limits_{1\le j\le n}\|\mu_j'\|_{\rm TV},
$$
where $\mu:=\mu_1\otimes\ldots\otimes\mu_n$.
\end{theorem}

\begin{proof}
Applying Theorem \ref{t1.3},
for an arbitrary function $\varphi\in C_0^\infty(\mathbb{R})$ with $\|\varphi\|_\infty\le1$,
we have
$$
\int \varphi'(\langle a, x\rangle)\, \mu(dx)
=
\int\tfrac{1}{(\beta_1-\alpha_1)\ldots(\beta_n-\alpha_n)}\Bigl[\int\limits_{\prod_{j=1}^n[\alpha_j,\beta_j]}
\varphi'(\langle a, x\rangle)\, dx\Bigr]\pi_1(d\alpha_1d\beta_1)\ldots\pi_n(d\alpha_nd\beta_n).
$$
Note that the inner integral can be represented as
$$
\int_{Q_n} \varphi'\bigl(\langle La, x\rangle + \langle a,c\rangle\bigr)\, dx,
$$
where $Q_n=[-1/2, 1/2]^n$ is the unit cube in $\mathbb{R}^n$,
$\ell_j=\beta_j-\alpha_j$, $L={\rm diag}(\ell_1,\ldots,\ell_n)$, and where
$c=2^{-1}(\alpha_1+\beta_1, \ldots, \alpha_n+\beta_n)$.
Let
$\theta = |La|^{-1}La$, let $\psi(t) = \varphi(|La|t+\langle a,c\rangle)$
and let $\rho_\theta$ be the distribution density of the  random variable $\langle\theta, x\rangle$
on the probability space $(\mathbb{R}^n, \mathcal{B}(\mathbb{R}^n), I_{Q_n}dx)$.
Then
\begin{multline*}
\int_{Q_n} \varphi'\bigl(\langle La, x\rangle + \langle a,c\rangle\bigr)\, dx
=
|La|^{-1}\int \psi'(t)\rho_\theta(t)\, dt
\\
\le
|La|^{-1}\|\rho_\theta'\|_{\rm TV}
=
|La|^{-1}2\max\rho_\theta=
2|La|^{-1}\rho_\theta(0),
\end{multline*}
since $\rho_\theta$ is even and logarithmically concave.
Note that
$\rho_\theta(0) = |\langle\theta\rangle^\bot\cap Q_n|\le \sqrt2$ due to
\cite{B-Cube1} (see also \cite{B-Cube2}).
Thus,
$$
\int \varphi'(\langle a, x\rangle)\, \mu(dx)
\le
2\sqrt2 \int\Bigl(\sum_{j=1}^n a_j^2\ell_j^2\Bigr)^{-1/2}\, \pi_1(d\alpha_1d\beta_1)\ldots\pi_n(d\alpha_nd\beta_n),
$$
where $\ell_j=\beta_j-\alpha_j$.
By the convexity of the function $t\to t^{-1/2}$ for $t>0$
we have
$$
\Bigl(\sum_{j=1}^n a_j^2\ell_j^2\Bigr)^{-1/2}
\le
\sum_{j=1}^n a_j^2 \ell_j^{-1}.
$$
We also recall that $\|\nu_{[\alpha_j,\beta_j]}'\|_{\rm TV} = 2\ell_j^{-1}$,
where $\nu_{[\alpha_j,\beta_j]}$ is the uniform distribution on the interval $[\alpha_j,\beta_j]$,
which implies that
\begin{multline*}
\int \varphi'(\langle a, x\rangle)\, \mu(dx)
\le
\sqrt2 \int\sum_{j=1}^na_j^2\|\nu_{[\alpha_j,\beta_j]}'\|_{\rm TV}\, \pi_1(d\alpha_1d\beta_1)\ldots\pi_n(d\alpha_nd\beta_n)
\\
=
\sqrt2\sum_{j=1}^na_j^2\|\mu_j'\|_{\rm TV}
\le
\sqrt2\max\limits_{1\le j\le n}\|\mu_j'\|_{\rm TV}.
\end{multline*}
The theorem is proved.
\end{proof}

We now proceed to the multidimensional setting.

\begin{theorem}
Let $k\in\mathbb{N}$. Then there is $C(k)$, depending only on $k$,
such that,
for any Skorohod differentiable
probability measures $\mu_1,\ldots, \mu_n$ on $\mathbb{R}$
and
for any mapping $F\colon\mathbb{R}^n\to\mathbb{R}^k$
of the form $F(x)=Ax$, where $A$ is a $k\times n$ matrix
with orthonormal rows $a_1,\ldots, a_k$,
one has
$$
\|\partial_e(\mu\circ F^{-1})\|_{\rm TV}\le
C(k)\max\limits_{1\le j\le n}\|\mu_j'\|_{\rm TV}\quad \forall e\in\mathbb{R}^k,\, |e|=1,
$$
where $\mu:=\mu_1\otimes\ldots\otimes\mu_n$.
\end{theorem}

\begin{proof}
Applying Theorem \ref{t1.3},
for any function $\varphi\in C_0^\infty(\mathbb{R}^k)$
with $\|\varphi\|_\infty\le1$
and any unit vector $e$,
one has
$$
\int \partial_e\varphi(Ax)\, \mu(dx)
=
\int\tfrac{1}{(\beta_1-\alpha_1)\ldots(\beta_n-\alpha_n)}\Bigl[\int\limits_{\prod_{j=1}^n[\alpha_j,\beta_j]}
\partial_e\varphi(Ax)\, dx\Bigr]\, \pi_1(d\alpha_1d\beta_1)\ldots\pi_n(d\alpha_nd\beta_n).
$$
The inner integral is equal to
$$
\int_{Q_n} \partial_e\varphi\bigl(ALx + Ac\bigr)\, dx,
$$
where $Q_n=[-1/2, 1/2]^n$ is the unit cube in $\mathbb{R}^n$,
$L={\rm diag}(\ell_1,\ldots,\ell_n)$, $\ell_j=\beta_j-\alpha_j$,
and  $c=((\alpha_1+\beta_1)/2,\ldots,(\alpha_n+\beta_n)/2)$.
Let $\sigma$ be the distribution of the random vector $(AL^2A^T)^{-1/2}ALx$
on the probability space $(\mathbb{R}^n, \mathcal{B}(\mathbb{R}^n), I_{Q_n}dx)$.
Note that $\sigma$ is log-concave and isotropic. Indeed, letting $B=AL$, for any $\theta\in\mathbb{R}^k$
we have
\begin{multline*}
\int \langle y,\theta\rangle^2\, \sigma(dy)=
\int_{Q_n} \langle(BB^T)^{-1/2}Bx, \theta\rangle^2\, dx
=
\int_{Q_n} \langle x, B^T(BB^T)^{-1/2}\theta\rangle^2\, dx
\\
=
\langle B^T(BB^T)^{-1/2}\theta,B^T(BB^T)^{-1/2}\theta\rangle
=
\langle (BB^T)^{-1/2}BB^T(BB^T)^{-1/2}\theta,\theta\rangle
=|\theta|^2.
\end{multline*}
Let $\rho$ be the density of the measure $\sigma$.
By Theorems \ref{t1.1} and \ref{t1.2} there is
$C_k$, depending only on $k$, and there are two absolute constants
$C$ and $c$ such that
$$
\rho(x)\le C_k^k e^{Ck - c|x|}.
$$
Thus, for every unit vector $h$
$$
\|D_h\sigma\|_{\rm TV}=2\int_{\langle h\rangle^\bot} \max\limits_t\rho(x+th)\, dx
\le
2C_k^ke^{Ck}\int_{\mathbb{R}^{k-1}} e^{-c\sqrt{x_1^2+\ldots+x_{k-1}^2}}\, dx_1\ldots dx_{k-1}=C_1(k).
$$
We now consider the function
$$
\psi(x):=\varphi\bigl((AL^2A^T)^{1/2}x + Ac\bigr)
$$
and
the vector
$$
\theta:= |(AL^2A^T)^{-1/2}e|^{-1}(AL^2A^T)^{-1/2}e.
$$
Note that
\begin{multline*}
\int\partial_\theta\psi\, d\sigma=
|(AL^2A^T)^{-1/2}e|^{-1}
\int (\partial_e\varphi)\bigl((AL^2A^T)^{1/2}y + Ac\bigr)\, \sigma(dy)
\\
=
|(AL^2A^T)^{-1/2}e|^{-1}
\int_{Q_n} \partial_e\varphi\bigl(ALx + Ac\bigr)\, dx
\end{multline*}
which implies the bound
$$
\int_{Q_n} \partial_e\varphi\bigl(ALx + Ac\bigr)\, dx
\le
C_1(k)|(AL^2A^T)^{-1/2}e|
\le
C_1(k)\|(AL^2A^T)^{-1/2}\|_{op},
$$
where for any matrix $B$ its operator norm is $\|B\|_{op}:=\sup\limits_{|h|=1}|Bh|$.
Let us now recall that
$$
\|(AL^2A^T)^{-1/2}\|_{op} =
(\inf\limits_{|h|=1} \langle (AL^2A^T)^{1/2}h,(AL^2A^T)^{1/2}h\rangle)^{-1/2}
=
\sup\limits_{|h|=1} (\langle AL^2A^Th,h\rangle)^{-1/2}.
$$
We note that the function $R_h(\ell_1^2,\ldots,\ell_n^2)=\langle AL^2A^Th,h\rangle$
is linear for any fixed unit vector $h$. Thus,
$$
R_h(\ell_1^2,\ldots,\ell_n^2) = k_1(h)\ell_1^2+\ldots+k_n(h)\ell_n^2.
$$
Moreover,
$$
k_1(h)+\ldots+k_n(h)=R_h(1,\ldots,1)=\langle A^Th, A^Th\rangle=1
$$
and
$$
k_j(h)=R_h(0,\ldots,1,\ldots,0)\ge0.
$$
By the convexity of the function $t\mapsto t^{-1/2}$ for $t>0$ we get
\begin{multline*}
\|(AL^2A^T)^{-1/2}\|_{op}
=
\sup\limits_{|h|=1} (k_1(h)\ell_1^2+\ldots+k_n(h)\ell_n^2)^{-1/2}
\\
\le
\sup\limits_{|h|=1} [k_1(h)\ell_1^{-1}+\ldots+k_n(h)\ell_n^{-1}]
=
\sup\limits_{|h|=1} R_h(\ell_1^{-1},\ldots,\ell_n^{-1})
=
\sup\limits_{|h|=1}
\langle AL^{-1}A^Th,h\rangle
\\
\le
{\rm tr}[AL^{-1}A^T]
=
{\rm tr}[A^TAL^{-1}]
=\sum_{j=1}^n\Bigl(\sum_{i=1}^k a_{j,i}^2\Bigr)\ell_j^{-1}.
\end{multline*}
Recall that $\|\nu_{[\alpha_j,\beta_j]}'\|_{\rm TV} = 2\ell_j^{-1}$. Hence
\begin{multline*}
\int \partial_e\varphi(Ax)\, \mu(dx)
\le
2^{-1}C_1(k) \int\sum_{j=1}^n\Bigl(\sum_{i=1}^k (a_{j,i})^2\Bigr)
\|\nu_{[\alpha_j,\beta_j]}'\|_{\rm TV}\, \pi_1(d\alpha_1d\beta_1)\ldots\pi_n(d\alpha_nd\beta_n)
\\
=
C_1(k)2^{-1}\sum_{j=1}^n\Bigl(\sum_{i=1}^k (a_{j,i})^2\Bigr)\|\mu_j'\|_{\rm TV}
\le
2^{-1}k C_1(k)\max\limits_{1\le j\le n}\|\mu_j'\|_{\rm TV},
\end{multline*}
which completes the proof.
\end{proof}

\begin{remark}
{\rm
We note that the  constant $C(k)$ obtained  in the previous theorem equals
$$
k C_k^ke^{Ck}\int_{\mathbb{R}^{k-1}} e^{-c\sqrt{x_1^2+\ldots+x_{k-1}^2}}\, dx_1\ldots dx_{k-1},
$$
where $C_k\sim k^{1/4}$. The integral above is equal to
$$
\frac{(k-1)\pi^{\frac{k-1}{2}}\Gamma(k-1)}{c^{k-1}\Gamma(\frac{k-1}{2}+1)}.
$$
Thus, $C(k)\sim (Ck)^{\frac{3k}{4}}$. The hyperplane conjecture asserts that
$C_k$ is actually independent of~$k$ and then $C(k)$ must be equivalent to $(Ck)^{k/2}$.
It is interesting to understand the actual dependence of $C(k)$ of $k$.
}
\end{remark}

\section{Polynomial images of general Skorohod differentiable measures}

In this section we study the regularity properties of polynomial images
of Skorohod differentiable measures on $\mathbb{R}^n$.

We start with the following technical lemma.

\begin{lemma}
Let $\mu$ be a probability measure on $\mathbb{R}^n$ Skorohod differentiable
along a unit vector $\theta\in\mathbb{R}^n$ and
let $f$ be a polynomial of degree $d$ on $\mathbb{R}^n$.
Consider the function
$$
F := \frac{f}{f^2+a^2}.
$$
Let $g\in C^\infty(\mathbb{R}^n)\cap L^\infty(\mathbb{R}^n)$ be a function such that $\partial_\theta g F\in L^1(\mu)$.
Then
$$
\int \partial_\theta g F\, d\mu\le (3(d+1)\pi + 1/2)\|D_\theta\mu\|_{\rm TV} a^{-1}\|g\|_\infty.
$$
\end{lemma}

\begin{proof}
Without loss of generality we can assume that $a=1$.
For every function $\psi\in C_0^\infty(\mathbb{R}^n)$
we have
\begin{equation*}
\int \partial_\theta\psi F\, d\mu =
- \int \psi\left[\frac{\partial_\theta f}{f^2+1}-\frac{2f^2\partial_\theta f}{(f^2+1)^2}\right]\, d\mu-
\int \psi(x)F(x)\, D_\theta\mu(dx).
\end{equation*}
The right-hand side is estimated from above by
$$
\|\psi\|_\infty\left(3\int\frac{|\partial_\theta f|}{f^2+1}d\mu_n + 2^{-1}\|D_\theta\mu\|_{\rm TV}\right).
$$
Let us now note that
\begin{multline*}
\int\frac{|\partial_\theta f|}{f^2+1}\, d\mu=
\int_{\langle\theta\rangle^\bot}\int_{\mathbb{R}}
\left|\frac{d}{dt}\mathrm{arctg}(f(y+t\theta))\right|\rho(y+t\theta)\, dtdy\\
\le (d+1)\pi\int_{\langle\theta\rangle^\bot}\sup_\tau\rho(y+\tau\theta)\, dy
\le(d+1)\pi \|D_\theta\mu\|_{\rm TV}.
\end{multline*}
Thus,
\begin{equation}\label{est1}
\int \partial_\theta\psi F\, d\mu\le(3(d+1)\pi + 1/2)\|D_\theta\mu\|_{\rm TV}\|\psi\|_\infty.
\end{equation}
Let $\psi_k(x) = \psi (x/k)$, where $\psi\in C_0^\infty(\mathbb{R}^n)$ is a function such that
$\psi(x) = 1$ for $|x|\le1$, $\psi(x)=0$ for $|x|>2$, and $\psi(x)\in [0, 1]$ for each $x$.
We now apply estimate (\ref{est1}) to the function $\psi_kg$:
$$
\int \partial_\theta(\psi_kg) F\, d\mu\le (3(d+1)\pi + 1/2)\|D_\theta\mu\|_{\rm TV}\|g\|_\infty.
$$
Note that
$$
\int \partial_\theta(\psi_kg) F\, d\mu = \int g\partial_\theta\psi_k F\, d\mu + \int \psi_k\partial_\theta g F\, d\mu,
$$
where the first term tends to zero and the second term tends to $\int \partial_\theta g F\, d\mu$.
Thus, we have obtained the desired estimate. The lemma is proved.
\end{proof}

\begin{corollary}\label{c1.1}
Let $\mu$ be a probability measure on $\mathbb{R}^n$ Skorohod differentiable
along a unit vector $\theta\in\mathbb{R}^n$
and let $f$ be a polynomial of degree $d$ on $\mathbb{R}^n$.
Then, for any function
$\varphi\in C_b^\infty(\mathbb{R})$ with $\|\varphi\|_\infty\le1$, any unit vector $\theta$,
and any positive number $\varepsilon$,
one has
$$
\int \varphi'(f)\partial_\theta f \frac{\partial_\theta f}{(\partial_\theta f)^2 + \varepsilon}d\mu_n
\le (3\pi d + 1/2)\|D_\theta\mu\|_{\rm TV}\varepsilon^{-1/2}.
$$
\end{corollary}

We are now ready to prove our second main result.

\begin{theorem}
Let $\mu$ be a probability measure on $\mathbb{R}^n$ Skorohod differentiable
along every vector $h\in\mathbb{R}^n$ and let $f$ be a polynomial of degree $d$ on $\mathbb{R}^n$ of the form
$f(x) = \sum_{j=0}^d B_j(x, \ldots, x)$, where each
$B_j\colon (\mathbb{R}^n)^j\to\mathbb{R}$ is a $j$-linear symmetric function.
Then, for an arbitrary function $\varphi\in C_b^\infty(\mathbb{R})$ with $\|\varphi\|_\infty\le1$,
one has
\begin{equation}\label{est2}
\int \varphi'(f)\, d\mu\le 12\pi \sup_{|\theta|=1}
\|D_\theta\mu\|_{\rm TV}\|B_d\|^{-1/d}\|\varphi'\|_\infty^{1-1/d}
\end{equation}
and the measure $\mu\circ f^{-1}$
possess a density from the Nikolskii--Besov class $B^{1/d}_{1,\infty}(\mathbb{R})$.
Moreover,
$$
\mu(f\in A)\le 12\pi\sup_{|\theta|=1}\|D_\theta\mu\|_{\rm TV}\|B_d\|^{-1/d} [\lambda(A)]^{1/d}.
$$
\end{theorem}

\begin{proof}
We will prove this theorem by induction on $d$.

The base of induction.
We first consider the case $d=1$. In this case $f(x) = \langle a, x \rangle + b = \sum_{k=1}^na_k x_k + b$
for some $a\in \mathbb{R}^n$ and $b\in\mathbb{R}$.
Note that for an arbitrary function $\varphi\in C_0^\infty(\mathbb{R})$ with $\|\varphi\|_\infty\le1$ one has
$$
\int\varphi'(f)\, d\mu=
-|a|^{-2}\int\varphi\Bigl(\sum_{k=1}^na_k x_k+b\Bigr)\, D_a\mu(dx)
\le|a|^{-2}\|D_a\mu\|_{\rm TV}\le |a|^{-1}\sup_{|\theta|=1}\|D_\theta\mu\|_{\rm TV}.
$$

The induction step.
Assume that estimate (\ref{est2}) is valid for every polynomial of degree not greater than $d$ and
let $f$ be a polynomial of degree $d+1$.
Note that for an arbitrary function $\varphi\in C_0^\infty(\mathbb{R})$ with $\|\varphi\|_\infty\le1$ we have
\begin{equation}\label{exp2}
\int\varphi'(f)\, d\mu =
\int\varphi'(f)\partial_\theta f
\frac{\partial_\theta f}{(\partial_\theta f)^2+\varepsilon}\, d\mu+
\varepsilon\int\varphi'(f(x))
\left((\partial_\theta f)^2+\varepsilon\right)^{-1}\, d\mu
\end{equation}
for an arbitrary number $\varepsilon>0$ and for an arbitrary unit vector $\theta\in \mathbb{R}^n$.
Let us estimate each term separately.
For the first term by Corollary \ref{c1.1} we have
$$
\int\varphi'(f)\partial_\theta f
\frac{\partial_\theta f}{(\partial_\theta f)^2+\varepsilon}d\mu_n \le
(3\pi d + 1/2)\sup_{|h|=1}\|D_h\mu\|_{\rm TV}\varepsilon^{-1/2}.
$$
By the equality $\partial_\theta f(x) = \sum_{j=1}^{d+1} jB_j(x,\ldots, x, \theta)$,
the induction hypothesis, and Lemma \ref{lem1.2},
the second term in (\ref{exp2}) can be estimated as follows:
\begin{multline*}
\varepsilon\int\varphi'(f)
((\partial_\theta f)^2+\varepsilon)^{-1}\, d\mu
\le\varepsilon\|\varphi'\|_\infty\int_0^{1/\varepsilon}\mu((\partial_\theta f)^2\le 1/t-\varepsilon)\, dt\\
=\varepsilon\|\varphi'\|_\infty\int_0^\infty(s+\varepsilon)^{-2}\mu(|\partial_\theta f|\le\sqrt{s})\, ds\\
\le\varepsilon\|\varphi'\|_\infty 12\pi\sup_{|h|=1}\|D_h\mu\|_{\rm TV}
(d+1)^{-1/d}\|B_{d+1}(\cdot,\ldots,\cdot, \theta)\|^{-1/d}2^{1/d}\int_0^\infty(s+\varepsilon)^{-2}s^{1/2d}\, ds\\
=\varepsilon^{1/2d}\|\varphi'\|_\infty \sup_{|h|=1}\|D_h\mu\|_{\rm TV}
\|B_{d+1}(\cdot,\ldots,\cdot, \theta)\|^{-1/d}2^{1/d}(d+1)^{-1/d}12\pi\int_0^\infty(s+1)^{-2}s^{1/2d}\, ds.
\end{multline*}
Thus, we have
\begin{multline*}
\int\varphi'(f)\, d\mu
\le\varepsilon^{-1/2}(3\pi d+1/2)\sup_{|h|=1}\|D_h\mu\|_{\rm TV}
\\
+
\varepsilon^{1/2d}\|\varphi'\|_\infty \sup_{|h|=1}\|D_h\mu\|_{\rm TV}
\|B_{d+1}(\cdot,\ldots,\cdot, \theta)\|^{-1/d}
2^{1/d}(d+1)^{-1/d}12\pi\int_0^\infty(s+1)^{-2}s^{1/2d}\, ds.
\end{multline*}
Setting
$$
\varepsilon= (3\pi d+1/2)^\frac{2d}{d+1}\Bigl(\|\varphi'\|_\infty
\|B_{d+1}(\cdot,\ldots,\cdot, \theta)\|^{-\frac{1}{d}}2^{\frac{1}{d}}
(d+1)^{-\frac{1}{d}}12\pi\int_0^\infty(s+1)^{-2}s^{\frac{1}{2d}}\, ds\Bigr)^{-\frac{2d}{d+1}},
$$
we get
$$
\int\varphi'(f)\, d\mu\le c_{d+1}\sup_{|h|=1}\|D_h\mu\|_{\rm TV}
\|B_{d+1}(\cdot,\ldots,\cdot, \theta)\|^{-1/(d+1)}\|\varphi'\|_\infty^{1-1/(d+1)}
$$
with
$$
c_{d+1} = (3\pi d+1/2)^\frac{1}{d+1}\Bigl(2^{1/d}(d+1)^{-1/d}12\pi
\int_0^\infty(s+1)^{-2}s^{1/2d}\, ds\Bigr)^{\frac{d}{d+1}}.
$$
Note that
$$
\int_0^\infty(s+1)^{-2}s^{1/2d}ds\le \Bigl(\int_0^\infty(s+1)^{-2}s^{1/2}ds\Bigr)^{1/d}  
= (\pi/2)^{1/d}
\le 2^{1/d}
$$
and also note that
$$
\frac{3\pi d+1/2}{d+1}\le3\pi.
$$
Thus, $c_{d+1}\le 12\pi$ and
taking the infimum over $\theta$ we get the desired bound.
Theorem is proved.
\end{proof}

From Lemma \ref{lem1.2} we get the following
result on the integrability of the distribution density.

\begin{corollary}
Let $\mu$ be a measure on $\mathbb{R}^n$ Skorohod differentiable
along every vector $h\in\mathbb{R}^n$ and let $f$ be a polynomial of degree $d$ on $\mathbb{R}^n$ of the form
$f(x) = \sum_{j=0}^d B_j(x, \ldots, x)$, where each
$B_j\colon (\mathbb{R}^n)^j\to\mathbb{R}$ is a $j$-linear symmetric function.
Then
the density $\rho_f$ of the measure
$\mu\circ f^{-1}$
is integrable to every power $p\in\bigl[1, \frac{d}{d-1}\bigr)$
and
$$
\|\rho_f\|_p\le c(p,d)\bigl(\sup_{|\theta|=1}
\|D_\theta\mu\|_{\rm TV} \bigr)^{d(1-1/p)}\|B_d\|^{1/p-1},
$$
where
$c(p,d) = \Bigl(p + p\bigl(\frac{d}{d-1}-p\bigr)^{-1}\Bigr)^{1/p}(12\pi)^{d(1-1/p)}$.
\end{corollary}

From \cite[Lemma 2.3]{Kos} (see also \cite[Theorem 3.2]{BKZ})
we get the following corollary.

\begin{corollary}
Let $\mu$ be a measure on $\mathbb{R}^n$ Skorohod differentiable
along every vector $h\in\mathbb{R}^n$ and let $f_1$ and $f_2$ be two polynomials of degree $d$ on $\mathbb{R}^n$
of the form
$f_i(x) = \sum_{j=0}^d B^i_j(x, \ldots, x)$, where each
$B^i_j\colon (\mathbb{R}^n)^j\to\mathbb{R}$ is a $j$-linear symmetric function, $i=1,2$.
Then
$$
\|\mu\circ f_1^{-1} - \mu\circ f_2^{-1}\|_{\rm TV}
\le C\|\mu\circ f_1^{-1} - \mu\circ f_2^{-1}\|_{\rm KR}^{1/(1+d)},
$$
where
$$
C=2\Bigl(1+12\pi\sup_{|\theta|=1}\|D_\theta\mu\|_{\rm TV}\bigl(\|B^1_d\|^{-1/d} + \|B^2_d\|^{-1/d}\bigr)\Bigr)
$$
and $\|\cdot\|_{\rm KR}$ is the Kantorovich--Rubinstein norm of a measure defined by
$$
\|\nu\|_{\rm KR} := \sup\biggl\{\int \varphi\, d\nu, \ \varphi\in C_0^\infty(\mathbb{R}^n),\ \|\varphi\|_\infty\le1,\ \|\varphi'\|_\infty\le1\biggr\}.
$$
\end{corollary}

\vskip .1in

The author is a Young
Russian Mathematics award winner and would like to thank its sponsors and jury.

The article was prepared within the framework of the
HSE University Basic Research Program
and funded by the Russian Academic Excellence Project '5-100'.

This research was supported by the RFBR
Grant 17-01-00662 and by the Foundation for the Advancement of Theoretical
Physics and Mathematics ``BASIS''.

\end{document}